\documentclass[11pt]{amsart}
\usepackage{amsmath, amssymb, mathtools, hyperref, graphicx, subfigure, paralist, enumitem, xcolor}
\usepackage[paper=a4paper,left=3cm,right=3cm,top=3cm,bottom=3cm]{geometry}

\newtheorem{theorem}{Theorem}[section]
\newtheorem{proposition}[theorem]{Proposition}

\newtheorem{lemma}[theorem]{Lemma}
\newtheorem{conjecture}[theorem]{Conjecture}
  \theoremstyle{definition}

\newtheorem{example}{Example}

  \theoremstyle{remark}
\newtheorem{remark}[theorem]{Remark}

\numberwithin{equation}{section}

\renewcommand{\qed}{\hfill $ \blacksquare $}
\newcommand{\myqed}{\hfill $\blacktriangle$}
\newcommand{ \sml }[1]{ #1 }
\newcommand{ \C }{ {\mathcal C} }

\newcommand{ \Z }{ \mathbb{Z} }
\newcommand{ \N }{ \mathbb{N} }
\newcommand{ \da }{ d_{\textnormal{a}} }
\newcommand{ \floor }[1]{ \left\lfloor #1 \right\rfloor }
\newcommand{ \ceil }[1]{ \left\lceil #1 \right\rceil }
\newcommand{ \norm }[1]{ {\lVert #1 \rVert}_{1} }
\newcommand{ \vol }{ \operatorname{Vol} }

\begin{document}

\title[Sidon Sets, Difference Sets, and Codes in $ A_n $ Lattices]
      {Sidon Sets, Difference Sets, and Codes in $ \boldsymbol{A_n} $ Lattices}

\author{Mladen~Kova\v{c}evi\'{c}}

\address{BioSense Institute, University of Novi Sad, 21000 Novi Sad, Serbia.}
\email{kmladen@uns.ac.rs}

\subjclass[2010]{Primary: 05B10, 05B40, 05B45, 94B25;
                 Secondary: 11B75, 11T71, 52C17, 52C22.}

\date{November 15, 2019.}

\keywords{Sidon set, difference set, prime power conjecture,
perfect code, tiling, group splitting, lattice packing,
asymmetric channels.}

\begin{abstract}
This chapter investigates the properties of (linear) codes in $ A_n $
lattices, the practical motivation for which is found in several communication
scenarios, such as asymmetric channels, sticky-insertion channels, bit-shift
channels, and permutation channels.
In particular, a connection between these codes and notions of difference
sets and Sidon sets in Abelian groups is demonstrated.
It is shown that the $ A_n $ lattice admits a linear perfect code of radius
$ 1 $ if and only if there exists an Abelian planar difference set of cardinality
$ n + 1 $.
Similarly, a direct link is given between linear codes of radius $ r $ in
the $ A_n $ lattice and Sidon sets of order $ 2r $ and cardinality $ n + 1 $.
Sidon sets of order $ 2r-1 $ are also represented geometrically in a similar way.
Apart from providing geometric intuition about Sidon sets, this interpretation
enables simple derivations of bounds on their parameters, which are either
equivalent to, or improve upon the known bounds.
In connection to the above, more general (non-planar) Abelian difference
sets and perfect codes of radius $ r $ are also discussed.%
\end{abstract}

\maketitle

\section{Introduction}

Packing in lattices is a geometric problem underlying error correction in many
information transmission and storage systems.
In this chapter we investigate packings in $ A_n $ lattices endowed with the
$ \ell_1 $ metric, a problem arising naturally in several communication scenarios.
The problem is also closely related to some well-known combinatorial objects
such as Sidon sets and difference sets and represents their geometric counterpart,
in a sense that will be made precise in the sequel.

In the rest of this section we introduce $ A_n $ lattices, describe some of
their properties that will be exploited later on, and explain the motivation
for this work by listing several communication settings for which the results
presented in this chapter are relevant.
In Sections \ref{sec:Bh}--\ref{sec:gendif} packings in $ A_n $ lattices are
studied in greater detail and, in particular, their connection with Sidon sets
and difference sets is demonstrated.
Most of the material presented herein appears in \cite{kovacevic+tan_it}.

\subsection{$ A_n $ lattice under $ \ell_1 $ metric}
\label{sec:An}

The $ A_n $ lattice is defined as
\begin{equation}
  A_n = \left\{ (x_0, x_1, \ldots, x_n) : x_i \in \mathbb{Z}, \sum_{i=0}^n x_i = 0 \right\}
\end{equation}
where $ \mathbb{Z} $ denotes the integers, as usual.
$ A_1 $ is equivalent to $ \mathbb{Z} $, $ A_2 $ to the hexagonal lattice,
and $ A_3 $ to the face-centered cubic lattice (see \cite{conway+sloane}).
The metric on $ A_n $ that we consider is essentially the $ \ell_1 $ (also
termed Manhattan, or taxi) distance:
\begin{equation}
\label{eq:metricd}
  d({\bf x}, {\bf y}) = \frac{1}{2} \norm{{\bf x} - {\bf y}}
                      = \frac{1}{2} \sum_{i = 0}^n | x_i - y_i | ,
\end{equation}
where $ {\bf x} = ( x_0, x_1, \ldots, x_n ) $, $ {\bf y} = ( y_0, y_1, \ldots, y_n ) $;
the constant $ 1/2 $ is introduced for convenience because $ \norm{{\bf x} - {\bf y}} $
is always even for $ {\bf x}, {\bf y} \in A_n $.
The metric $ d $ also represents the graph distance in $ A_n $.
Namely, if $ \Gamma(A_n) $ is a graph with the vertex set $ A_n $ and with
edges joining neighboring points (points at distance $ 1 $ under $ d $), then
$ d({\bf x}, {\bf y}) $ is the length of the shortest path between $ {\bf x} $
and $ \bf y $ in $ \Gamma(A_n) $.

Ball of radius $ 1 $ around $ {\bf x} \in A_n $ contains $ 2 {n+1 \choose 2} + 1 = n^2 + n + 1 $
points of the form $ {\bf x} + {\bf f}_{i,j} $, where $ {\bf f}_{i,j} $ is
the vector having $ 1 $ at the $ i $'th coordinate, $ -1 $ at the $ j $'th
coordinate, and zeros elsewhere (with the convention $ {\bf f}_{i,i} = {\bf 0} $).
The convex interior of the points in this ball forms a highly symmetrical
polytope having the following interesting property, among many others -- the
distance between any vertex and the center is equal to the distance between
any two neighboring vertices.
Ball of radius $ r $ around $ {\bf x} \in A_n $ contains all the points
with integral coordinates in the convex interior of $ \{ {\bf x} + r {\bf f}_{i,j} \}_{i,j} $.

\begin{figure}[h]
\centering
  \includegraphics[width=0.5\columnwidth]{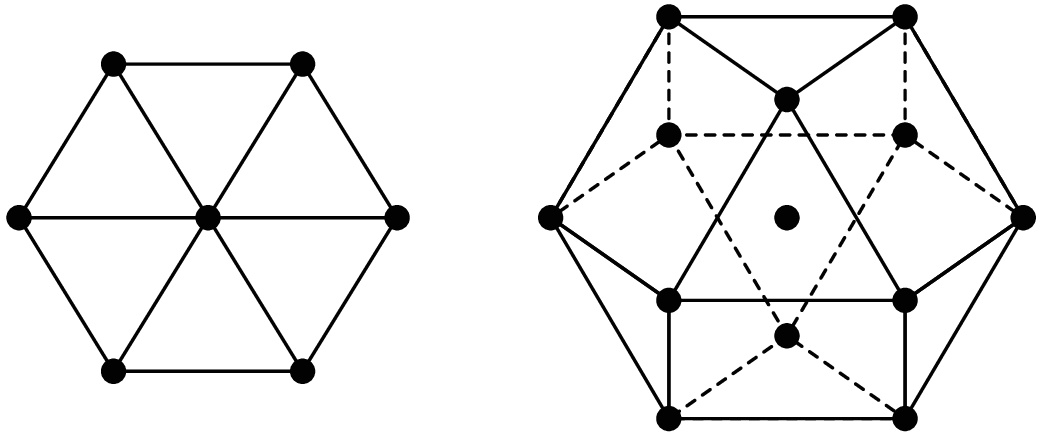}
  \caption{Ball of radius $ 1 $ in $ (A_2, d) $ -- hexagon, and in $ (A_3, d) $ -- cuboctahedron.}
\label{fig:balls}
\end{figure}%

An error-correcting code of radius $ r $ in $ (A_n, d) $ is a subset of
$ A_n $ with the property that balls of radius $ r $ centered at points
of this subset are disjoint.
A code is said to be linear if it is a sublattice of $ A_n $ (i.e., a subset
of $ A_n $ closed under addition and subtraction).
The minimum distance of a code $ \C \subseteq A_n $ with respect to the
metric $ d $ is denoted by $ d(\C) $.

For the purpose of studying packing problems, it is usually simpler to
visualize $ \mathbb{Z}^n $ instead of an arbitrary lattice.
In our case there is a simple mapping that makes the transition to $ \mathbb{Z}^n $
and back very easy.
For $ {\bf x} = (x_1, \ldots, x_n), {\bf y} = (y_1, \ldots, y_n) \in \mathbb{Z}^n $,
define the metric
\begin{equation}
  \da({\bf x}, {\bf y}) =
    \max \left\{ \sum_{ i :\, x_i > y_i } (x_i - y_i) , \sum_{ i :\, x_i < y_i } (y_i - x_i) \right \} .
\end{equation}
This distance is of interest in the theory of codes for asymmetric channels
\cite[Ch.\ 2.3 and 9.1]{klove}, which is why subscript `a' is used to denote it.

\begin{lemma}
\label{th:isometry}
  $ (A_n, d) $ is isometric to $ (\mathbb{Z}^n, \da) $.
\end{lemma}
\begin{proof}
  For $ {\bf x} = (x_0, x_1, \ldots, x_n) $, denote $ {\bf x}' = (x_1, \ldots, x_n) $.
The mapping $ {\bf x} \mapsto {\bf x}' $ is the desired isometry.
Just note that, for $ {\bf x}, {\bf y} \in A_n $,
\begin{equation}
  d({\bf x}, {\bf y}) = \sum_{\substack{ i = 0 \\ x_i > y_i}}^n (x_i - y_i)
                      = \sum_{\substack{ i = 0 \\ x_i < y_i}}^n (y_i - x_i)
\end{equation}
because $ \sum_{i=0}^n x_i = \sum_{i=0}^n y_i = 0 $, and then by examining the
cases $ x_0 \lessgtr y_0 $ it follows that
\begin{equation}
  d({\bf x}, {\bf y})
    =  \max \left\{ \sum_{\substack{ i = 1 \\ x_i > y_i}}^n (x_i - y_i) ,
                  \sum_{\substack{ i = 1 \\ x_i < y_i}}^n (y_i - x_i) \right\}
    = \da({\bf x}', {\bf y}') .
\end{equation}
Furthermore, the mapping $ {\bf x} \mapsto {\bf x}' $ is a bijection between
$ A_n $ and $ \Z^n $.
\end{proof}

Hence, packing and similar problems in $ (A_n, d) $ are equivalent
to those in $ (\mathbb{Z}^n, \da) $.
Balls in $ (\mathbb{Z}^n, \da) $ are ``distorted'' versions of the ones
in $ (A_n, d) $ (see Fig.\ \ref{fig:Z2tile}).
For example, the ball of radius $ r $ around $ \bf 0 $ in $ (\mathbb{Z}^n, \da) $
contains the points in $ \mathbb{Z}^n $ whose positive coordinates sum to
$ \leq\! r $, and negative to $ \geq\! -r $.

We shall also need the following generalization of a ball in $ (\mathbb{Z}^n, \da) $:
\begin{equation}
\label{eq:Sn}
  S_n(r^\sml{+}, r^\sml{-}) =
   \left\{ {\bf x} \in \mathbb{Z}^n :
       \sum_{i \, : \, x_i > 0} x_i \leq r^\sml{+} ,
        \sum_{i \, : \, x_i < 0} | x_i | \leq r^\sml{-} \right\} ,
\end{equation}
where $ r^\sml{+}, r^\sml{-} \geq 0 $.
This set is an anticode \cite{ahlswede} of maximum distance $ r^\sml{+} + r^\sml{-} $,
i.e., a subset of $ \Z^n $ of diameter $ r^\sml{+} + r^\sml{-} $.
Due to the symmetry of the set $ S_n(r^\sml{+}, r^\sml{-}) $ in the
parameters $ r^\sml{+}, r^\sml{-} $ (see Fig.\ \ref{fig:S2}), we shall assume
henceforth that $ r^\sml{-} \leq r^\sml{+} $.
For $ r^\sml{+} = r^\sml{-} = r $, $ S_n(r) \equiv S_n(r,r) $ is a ball of
radius $ r $ around $ \bf 0 $ in $ (\mathbb{Z}^n, \da) $.

\begin{figure}[h]
 \centering
  \includegraphics[width=0.85\columnwidth]{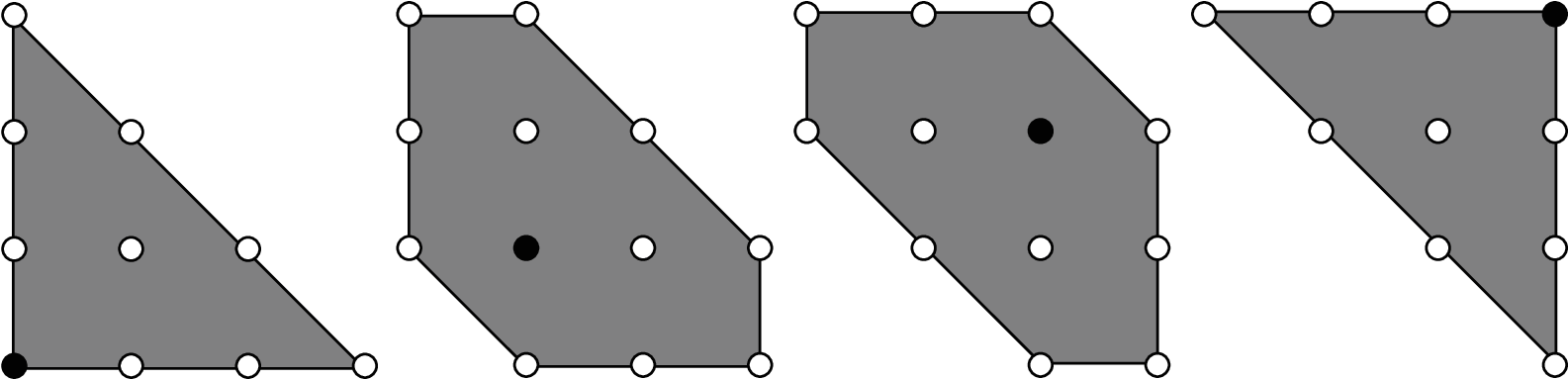}
\caption{Sets $ S_2(r^\sml{+},r^\sml{-}) $ for $ r^\sml{+} + r^\sml{-} = 3 $
         (black dot denotes the origin).}
\label{fig:S2}
\end{figure}%

\begin{lemma}
\label{th:Anball}
 The cardinality of the set $ S_n(r^\sml{+}, r^\sml{-}) $ is
\begin{equation}
\label{eq:A3ball}
  | S_n(r^\sml{+}, r^\sml{-}) | =
     \sum_{ m = 0 }^{ \min\{ n, r^\sml{+} \} }  { n \choose m }
                                                { r^\sml{+} \choose m }
                                                { r^\sml{-} + n - m \choose n - m } .
\end{equation}
\end{lemma}
\begin{proof}
  Consider the vectors in $ S_n(r^\sml{+}, r^\sml{-}) $ having $ m $ strictly
positive coordinates, $ m \in \{ 0, \ldots, n \} $.
These coordinates can be chosen in $ n \choose m $ ways.
For each such choice, the ``mass'' $ \leq\! r^\sml{+} $ can be distributed over
them in $ \sum_{t=m}^{r^\sml{+}} {t - 1 \choose m-1} = {r^\sml{+} \choose m} $
ways (think of placing $ t \leq r^\sml{+} $ balls into $ m $ bins, where at least
one ball is required in each bin).
Likewise, the mass $ \leq\! r^\sml{-} $ can be distributed over the remaining
coordinates in
$ \sum_{t=0}^{r^\sml{-}} {t +n-m -1 \choose n-m-1} = {r^\sml{-} + n-m \choose n-m } $
ways.
\end{proof}

\subsection{Applications}
\label{sec:motivation}

Even though there are no practical situations where the code space is the
entire $ \Z^n $ or $ A_n $ lattice, it is the underlying geometry of the
problem that provides the motivation for studying codes in such infinite spaces.
In situations where the code space is, for example, a restriction of the $ A_n $
lattice, the corresponding restrictions of dense packings in $ (A_n, d) $ will
be good codes for the original problem, at least in some asymptotic regimes.
In particular, the results presented in this chapter have been used for
constructing error-correcting codes and bounding their cardinality in
the context of permutations channels \cite{kovacevic+tan_it, kovacevic+vukobratovic_desi},
bit-shift channels \cite{kovacevic2}, and duplication/sticky-insertion
channels \cite{kovacevic+tan_clet}.
Codes under the metric $ \da $ are of relevance in several other scenarios
as well, most notably the so-called asymmetric channels \cite{klove, varshamov}.

\section{Sidon sets as linear codes in $ A_n $ lattices}
\label{sec:Bh}

In this section, linear codes in $ (A_n, d) $ and, more generally, lattice
packings of the set $ S_n(r^\sml{+}, r^\sml{-}) $ in $ \mathbb{Z}^n $, are
studied.
The connection to the so-called Sidon sets is described, and bounds on the
packing radius $ r $ and the dimension $ n $ derived.

\subsection{Sidon sets}

Let $ G $ be an Abelian group%
\footnote{Only Abelian groups are treated in the chapter, this is understood
even if not explicitly stated.}
of order $ v $, written additively.
A set $ B = \{ b_0, b_1, \ldots, b_{k-1} \} \subseteq G $ is said to be a
\emph{Sidon set} of order $ h $ (or a \emph{$ B_h $ set}) if
the sums $ b_{i_1} + \cdots + b_{i_h} $, $ 0 \leq i_1 \leq \cdots \leq i_h \leq k-1 $,
are all different.
$ B_2 $ sets were introduced by Sidon%
\footnote{Though there were some earlier appearances of the problem, see
\cite{obryant}.}
\cite{sidon}, and a construction of optimal such sets in
$ \mathbb{Z}_v \equiv \mathbb{Z} / v\mathbb{Z} $ was given by Singer \cite{singer}
for $ k - 1 $ a prime power and $ v = k^2 - k + 1 $.
Bose and Chowla \cite{bose+chowla} gave a construction of $ B_h $ sets in
$ \mathbb{Z}_v $ for arbitrary $ h \geq 1 $ when:
\begin{itemize}
\item
$ k $ is a prime power and $ v = k^h - 1 $,
\item
$ n = k - 1 $ is a prime power and $ v = \left( n^{h+1} - 1 \right) / (n - 1) $.
\end{itemize}
Since these pioneering papers, research in this area of combinatorial number
theory has been extensive, see \cite{obryant} for references.
It has also found numerous applications in coding theory; see, e.g.,
\cite{barg, derksen, graham+sloane, varshamov}.

As indicated above, $ k = n + 1 $ will denote the cardinality of the $ B_h $
set in question.
Note that if $ \{ b_0, b_1, \ldots, b_n \} $ is a $ B_h $ set, then so
is $ \{ 0, b_1-b_0, \ldots, b_n-b_0 \} $, and vice versa;
we shall therefore assume in the sequel that $ b_0 = 0 $ without loss in
generality.
With this convention, the defining property of $ B_h $ sets is that all the
sums $ b_{i_1} + \cdots + b_{i_u} $, for arbitrary $ 0 \leq u \leq h $ and
$ 1 \leq i_1 \leq \cdots \leq i_u \leq n $, are different.

\subsection{Sidon sets and packings in $ (\mathbb{Z}^n, \da) $}

If $ S, \mathcal{L} \subseteq \mathbb{Z}^n $, $ \mathcal{L} $ a lattice,
we say that $ S $ \emph{packs} $ \mathbb{Z}^n $ with lattice $ \mathcal{L} $
if the translates $ S + {\bf x} $ and $ S + {\bf y} $ are disjoint for every
$ {\bf x}, {\bf y} \in \mathcal{L} $, $ {\bf x} \neq {\bf y} $.
In particular, we are interested in packings by the set $ S_n(r^\sml{+}, r^\sml{-}) $
defined in \eqref{eq:Sn}.
In this terminology, a linear code of radius $ r $ is a lattice packing by
the balls $ S_n(r, r) $.
The following theorem states that such packings are, in a sense, geometric
realizations of $ B_h $ sets.

\begin{theorem}
\label{th:Bh}
  Let $ h \geq 1 $ and $ r^\sml{+}, r^\sml{-} \geq 0 $ be integers satisfying
$ r^\sml{+} + r^\sml{-} = h $.
\begin{enumerate}[leftmargin=7mm]
\item[(a)]
Assume that $ B = \{ 0, b_1, \ldots, b_n \} $ is a $ B_h $ set in an Abelian
group $ G $ of order $ v $, and that $ B $ generates $ G $.
Then $ S_n(r^\sml{+}, r^\sml{-}) $ packs $ \mathbb{Z}^n $ with lattice
\begin{equation}
\label{eq:Bhlattice}
  {\mathcal L} = \left\{ {\bf x} \in \mathbb{Z}^n : \sum_{i=1}^{n} x_i b_i = 0 \right\} ,
\end{equation}
and $ G $ is isomorphic to $ \mathbb{Z}^n / {\mathcal L} $ (here $ x_i b_i $
denotes the sum in $ G $ of $ |x_i| $ copies of $ b_i $ if $ x_i > 0 $, and
of $ -b_i $ if $ x_i < 0 $).
\item[(b)]
Conversely, if $ S_n(r^\sml{+}, r^\sml{-}) $ packs $ \mathbb{Z}^n $ with lattice
$ \mathcal{L}' $, then the group $ G = \mathbb{Z}^n / \mathcal{L}' $ contains a
$ B_h $ set of cardinality $ n + 1 $ that generates $ G $.
\end{enumerate}
\end{theorem}
\begin{proof}
  The claim is an instance of the familiar group-theoretic formulation of
lattice packing/tiling problems \cite{stein67, hamaker, stein74, galovich,
hickerson, stein+szabo}
(the formulation in \cite{stein84} is the one we used here), so we only
sketch the proof using coding-theoretic terminology.
The condition that the translates of $ S_n(r^\sml{+}, r^\sml{-}) $ are
disjoint means that the error-vectors from $ S_n(r^\sml{+}, r^\sml{-}) $
are correctable and have different syndromes.
Since positive coordinates of these vectors sum to $ t \leq r^\sml{+} $ and
negative to $ -s \geq -r^\sml{-} $, we see from \eqref{eq:Bhlattice} that
the syndromes are of the form
$ b_{i_1} + \cdots + b_{i_t} - b_{j_1} - \cdots - b_{j_s} $.
Now just note that all the sums $ b_{i_1} + \cdots + b_{i_u} $ are
different, where $ u $ goes through $ \{0, 1, \ldots, h\} $ and
$ 1 \leq i_1 \leq \cdots \leq i_u \leq n $, if and only if all linear
combinations of the form
$ b_{i_1} + \cdots + b_{i_t} - b_{j_1} - \cdots - b_{j_s} $ are different,
where $ t $ goes through $ \{0, 1, \ldots, r^\sml{+}\} $, $ s $ through
$ \{0, 1, \ldots, r^\sml{-}\} $, and $ 1 \leq i_1 \leq \cdots \leq i_t \leq n $,
$ 1 \leq j_1 \leq \cdots \leq j_s \leq n $.
Hence the need for a $ B_h $ set.
\end{proof}

In other words, the code \eqref{eq:Bhlattice} defined by a $ B_{r^\sml{+} + r^\sml{-}} $
set with $ n + 1 $ elements is capable of correcting all error-vectors from
the set $ S_n(r^\sml{+}, r^\sml{-}) $, and the cardinality of the Voronoi region
of an arbitrary codeword is $ v $.
In particular, $ B_{2r} $ sets correspond in a direct way to error-correcting
codes of radius $ r $ in $ (\mathbb{Z}^n, \da) $, or equivalently in $ (A_n, d) $.
Note also that the packing lattice $ {\mathcal L} $ does not depend on the
particular values of $ r^\sml{+}, r^\sml{-} $, but only on their sum.
However, the cardinality of the ``decoding region'' $ S_n(r^\sml{+}, r^\sml{-}) $
varies with $ r^\sml{+} $, and so does the packing density
$ \frac{1}{v} | S_n(r^\sml{+}, r^\sml{-}) | $
(the fraction of the space covered).

\begin{example}
  The set $ \{ (0,0), (1,1), (0,5) \} \subseteq \mathbb{Z}_2 \times \mathbb{Z}_6 $
is a $ B_3 $ set.
The corresponding packing in $ \mathbb{Z}^2 $ is illustrated in Fig.~\ref{fig:S2_tiling}.
It is in fact a perfect packing, i.e., a tiling; this will be further discussed
in Section \ref{sec:perfect}.
\myqed
\end{example}

\begin{figure}[h]
 \centering
  \includegraphics[width=0.77\columnwidth]{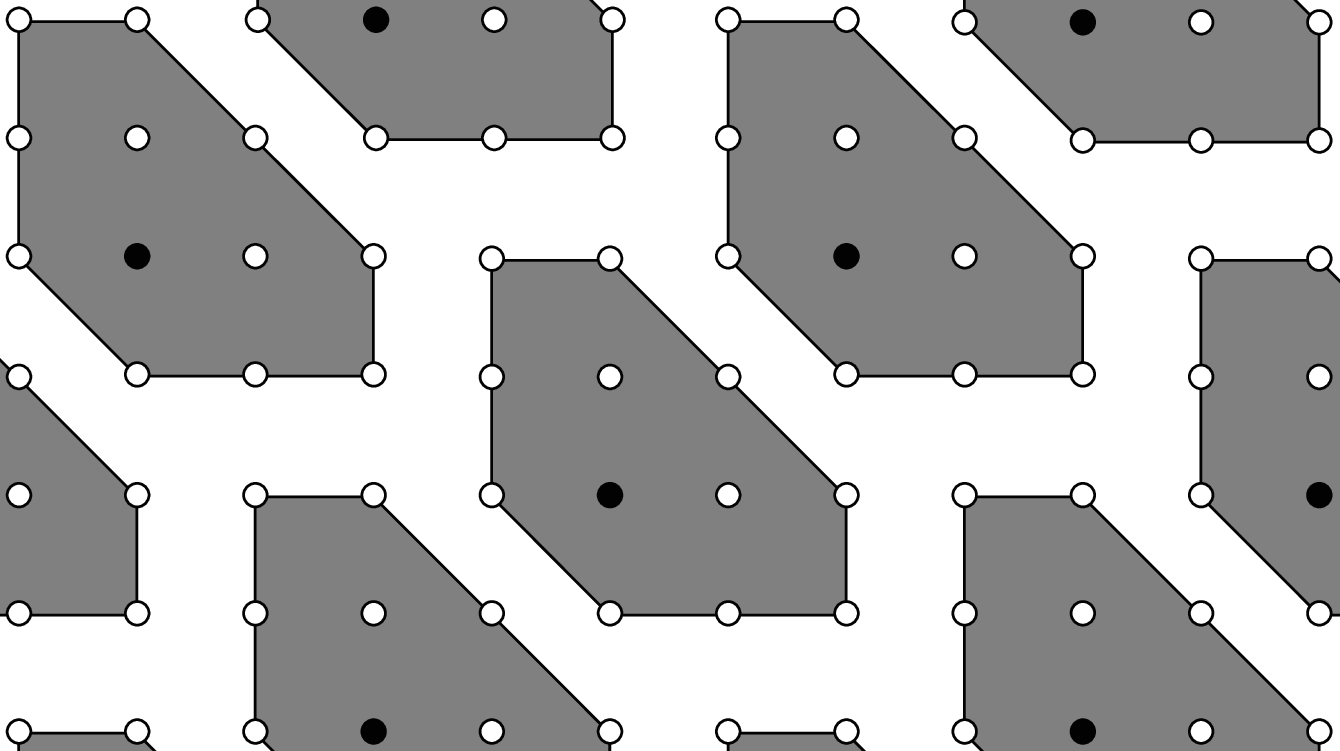}
\caption{Tiling of $ \mathbb{Z}^2 $ by the set $ S_2(2,1) $.}
\label{fig:S2_tiling}
\end{figure}%

\subsection{Bounds on the packing radius $ r $ and the dimension $ n $}
\label{sec:bounds}

We now make several simple observations about the density of the lattice
packings in $ (\mathbb{Z}^n, \da) $, and derive bounds on $ r $ and $ n $.
For easier comparison with the known bounds, we shall state them in terms
of the parameters of $ B_h $ sets.
Recalling that the dimension of the space $ n $ is related to the cardinality
of the $ B_h $ set $ k $ ($ k = n + 1 $), and the radius of the code $ r $
to the parameter $ h $ ($ h = 2r $), one can easily restate the bounds in
terms of these geometric quantities.

A major part of research on $ B_h $ sets is focused on determining maximum
cardinality of such sets for given $ h, v $, or the asymptotic behavior of
$ k $ for given $ h $ and for $ v \to \infty $.
Despite significant research efforts, however, determining tight bounds
remains an open problem.
Let $ f_h(v) $ denote the maximum cardinality of a $ B_h $ set in an
\emph{arbitrary} Abelian group of order $ v $,
$ h_k(v) $ the largest $ h $ for which there is a $ B_h $ set of cardinality
$ k $ in some Abelian group of order $ v $, and $ \phi(h, k) $ the order
(cardinality) of the smallest Abelian group containing a $ B_h $ set of
cardinality $ k $.
This notation is from \cite{jia}, except the functions are defined here to
include all finite Abelian groups, rather than just the cyclic ones.
The following statement illustrates how nontrivial bounds on these quantities
can be derived in a straightforward way using Theorem \ref{th:Bh}.

\begin{theorem}
\label{th:bounds}
  For all $ h \geq 1 $ and $ k \geq \ceil{h/2} $,
\begin{equation}
\label{eq:phik}
  \phi(h, k) > \frac{ (k - \ceil{h/2})^h }{ \ceil{h/2}! \floor{h/2}! } .
\end{equation}
For all $ h \geq 1 $ and $ v \geq 2 $,
\begin{equation}
\label{eq:fh}
  f_{h}(v)   < v^\frac{1}{h} \cdot \left( \floor{h/2}!\ceil{h/2}! \right)^\frac{1}{h} + \ceil{h/2} .
\end{equation}
For all $ k = n + 1 \geq 2 $ and $ h \geq 2k - 4 $,
\begin{equation}
\label{eq:phih}
  \phi(h, k) > \frac{ (h - 2n + 2)^n }{ n! } \frac{ {2n \choose n} }{2^n} .
\end{equation}
For all $ k = n + 1 \geq 2 $ and $ v \geq 2 $,
\begin{equation}
\label{eq:hk}
  h_k(v)  <  v^\frac{1}{n} \cdot 2 \left( (n!)^3 /(2n)! \right)^\frac{1}{n}  + 2n - 2 .
\end{equation}
\end{theorem}
\begin{proof}
  Suppose that $ \{ 0, b_1, \ldots, b_n \} $ is a $ B_h $ set in a group $ G $ of
order $ v $, and $ {\mathcal L} $ the corresponding lattice, see \eqref{eq:Bhlattice}.
By Theorem \ref{th:Bh}, $ G \cong \mathbb{Z}^n / {\mathcal L} $ and the translations
of the set $ S_n(r^\sml{+}, r^\sml{-}) $ by vectors in $ {\mathcal L} $ are disjoint.
Therefore, for $ n \geq r^\sml{+} $,
\begin{equation}
\begin{aligned}
 v  &\geq  | S_n(r^\sml{+}, r^\sml{-}) |  \\
    &=     \sum_{ m = 0 }^{ r^\sml{+} }  { n \choose m }
                                         { r^\sml{+} \choose m }
                                         { r^\sml{-} + n - m \choose n - m }  \\
    &\geq \sum_{ m = 0 }^{ r^\sml{+} } \frac{ (n - m + 1)^m }{ m! } { r^\sml{+} \choose m } \frac{ (n - m + 1)^{r^\sml{-}} }{ r^\sml{-}! }  \\
    &> \frac{ (n - r^\sml{+} + 1)^h }{ r^\sml{+}! \ r^\sml{-}! } ,
\end{aligned}
\end{equation}
where we used $ {n \choose m} \geq \frac{(n-m+1)^m}{m!} $, and the last inequality
is obtained by keeping only the summand $ m = r^\sml{+} $.
This implies \eqref{eq:phik} and \eqref{eq:fh} by taking $ r^\sml{+} = \ceil{h/2} $
(this choice minimizes the denominator $ r^\sml{+}! r^\sml{-}! $).

Similarly, if we let $ 0 \leq r^\sml{+} - n \leq r^\sml{-} $, then
\begin{equation}
\begin{aligned}
 v  &\geq  | S_n(r^\sml{+}, r^\sml{-}) |  \\
    &=     \sum_{ m = 0 }^{ n }  { n \choose m }
                                 { r^\sml{+} \choose m }
                                 { r^\sml{-} + n - m \choose n - m }  \\
    &\geq \sum_{ m = 0 }^{ n } { n \choose m } \frac{ (r^\sml{+} - m + 1)^m }{ m! } \frac{ (r^\sml{-} + 1)^{(n-m)} }{ (n-m)! }  \\
    &=    \frac{1}{n!} \sum_{ m = 0 }^{ n } { n \choose m }^2 (r^\sml{+} - m + 1)^m (r^\sml{-} + 1)^{(n-m)}  \\
    &> \frac{ (r^\sml{+} - n +1)^n }{ n! } {2n \choose n} .
\end{aligned}
\end{equation}
In the last step we used the assumption $ r^\sml{-} \geq r^\sml{+} - n $ and
the identity $ \sum_{m=0}^n {n \choose m}^2 = {2n \choose n} $.
When $ r^\sml{+} = \ceil{h/2} $ we get \eqref{eq:phih} and \eqref{eq:hk}.
\end{proof}

The above derivations, apart from being elementary, have the advantage of
being valid for all finite Abelian groups.
The bounds in \eqref{eq:phik} and \eqref{eq:fh} are the same as the known
bounds from \cite[Thm 1]{jia} and \cite[Thm 2]{chen}, but with explicit
error terms, while the bounds \eqref{eq:phih} and \eqref{eq:hk} improve
on the known bounds \cite[Thm 1]{jia} by a factor of
$ 2^{-n}{2n \choose n} > 1 $, $ n > 1 $.
The lower bound on $ h_k(v) $ stated in \cite[Thm 1(v)]{jia} can also be
derived by an easy geometric argument.
Namely, since the set $ S_n(r^\sml{+}, r^\sml{-}) $ is contained in a
hypercube $ \{ -r^\sml{-}, \ldots, 0, \ldots, r^\sml{+} \}^n $ of size
$ (h + 1)^n $, and since this hypercube tiles $ \mathbb{Z}^n $, we conclude
that
\begin{equation}
  \phi(h, k)  \leq  (h + 1)^n ,
\end{equation}
which gives the desired bound.
A significant improvement of this bound was given in \cite{kovacevic+tan_sidma}
for sufficiently large $ h $, by demonstrating a connection between $ B_h $
sets with $ h \to \infty $ and lattice packings of simplices in Euclidean spaces.
Namely, it was shown there that, for every fixed $ n = k - 1 $ and $ \epsilon > 0 $,
and for $ h \geq \underline{h}(n, \epsilon) $,
\begin{equation}
  \phi(h, k)  \leq  (1+\epsilon) h^n \frac{\binom{2n}{n}}{2 \, n!} .
\end{equation}

\section{Perfect codes in $ (A_n, d) $}
\label{sec:perfect}

A code in a given discrete metric space is said to be $ r $-\emph{perfect} if
balls of radius $ r $ around the codewords are disjoint and cover the entire
space.
Geometrically speaking, these are the best packings that one can have in the
sense that their packing density is equal to $ 1 $.
It is therefore important to study their existence, and methods of construction
when they do exist.
Note that, by Theorem \ref{th:Bh}, linear $ r $-perfect codes in $ (A_n, d) $
correspond to $ B_{2r} $ sets of cardinality $ n + 1 $ in an Abelian group of
order $ v = |S_n(r)| $; these $ B_{2r} $ sets are themselves called perfect.

\subsection{$ 1 $-Perfect codes and planar difference sets}
\label{sec:planar}

Linear $ 1 $-perfect codes in $ (A_n, d) $ correspond to $ B_2 $ sets (Sidon
sets of order $ 2 $) of cardinality $ k = n + 1 $ in an Abelian group $ G $ of
order $ v = n^2 + n + 1 $.
Such sets are better known in the literature as \emph{planar} (or \emph{simple})
\emph{difference sets}.
Note that all the sums $ b_i + b_j $ are different, up to the order of the
summands, if and only if all the differences $ b_i - b_j $ for $ i \neq j $
are different.
The additional requirement for difference sets, compared to $ B_2 $ sets, is
that every nonzero element of the group can be expressed as such a difference,
which is equivalent to saying that the order of the group is
$ v = k(k-1) + 1 = n^2 + n + 1 $.
The \emph{order} of a planar difference set $ D $ of cardinality $ k = n + 1 $
is defined as $ k - 1 = n $.

Planar difference sets and their generalizations (see Section \ref{sec:gendif})
are very well-studied, and a large body of literature is devoted to the
investigation of their properties \cite{designs}.
One of the most familiar problems in the area concerning the existence of
these objects for specific sets of parameters is the so-called \emph{prime
power conjecture} \cite[Conj. 7.5, p. 346]{designs} which states that a planar
difference set of order $ n $ exists if and only if $ n $ is a prime power
(counting $ n = 1 $ as a prime power).
Existence of such sets for $ n = p^m $, $ p $ prime, $ m \in \N $, was
demonstrated by Singer \cite{singer}, but the necessity of this condition
remains an open problem for nearly eight decades.
Difference sets have also been applied in communications and coding theory
in various settings, see for example \cite{ding, atkinson, lam+sarwate}.

The following claim is a slight modification of Theorem \ref{th:Bh}.

\begin{theorem}
\label{thm:1perfect}
  There exists an Abelian planar difference set of order $ n $ if and only if
the space $ (A_n, d) $ admits a linear $ 1 $-perfect code.
\hfill \qed
\end{theorem}

Existence of such codes when $ n $ is a prime power follows from the existence
of the corresponding difference sets \cite{singer}, but the necessity of this
condition is open and is equivalent to the prime power conjecture.

\begin{conjecture}[{Prime power conjecture}]
\label{conj:ppc}
  There exists a linear $ 1 $-perfect code in $ (A_n, d) $ (or, equivalently,
in $ (\mathbb{Z}^n, \da) $) if and only if $ n $ is a prime power.
\myqed
\end{conjecture}

A stronger conjecture would claim the above even for nonlinear codes.

\begin{example}
  Consider a planar difference set $ D = \{0, 1, 3, 9\} \subset \mathbb{Z}_{13} $.
The corresponding $ 1 $-perfect code in $ (A_3, d) $ is illustrated in Fig.\ \ref{fig:A3perfect}.
The figure shows the intersection of $ A_3 $ with the plane $ x_0 = 0 $;
the intersections of a ball of radius $ 1 $ in $ A_3 $ with the planes $ x_0 = \text{const} $
are shown in Fig.\ \ref{fig:ball3color} for clarification.
\myqed
\end{example}

\begin{figure}[h]
 \centering
  \subfigure[The code viewed in the plane $ x_0 = 0 $.]
  {
   \includegraphics[width=0.77\columnwidth]{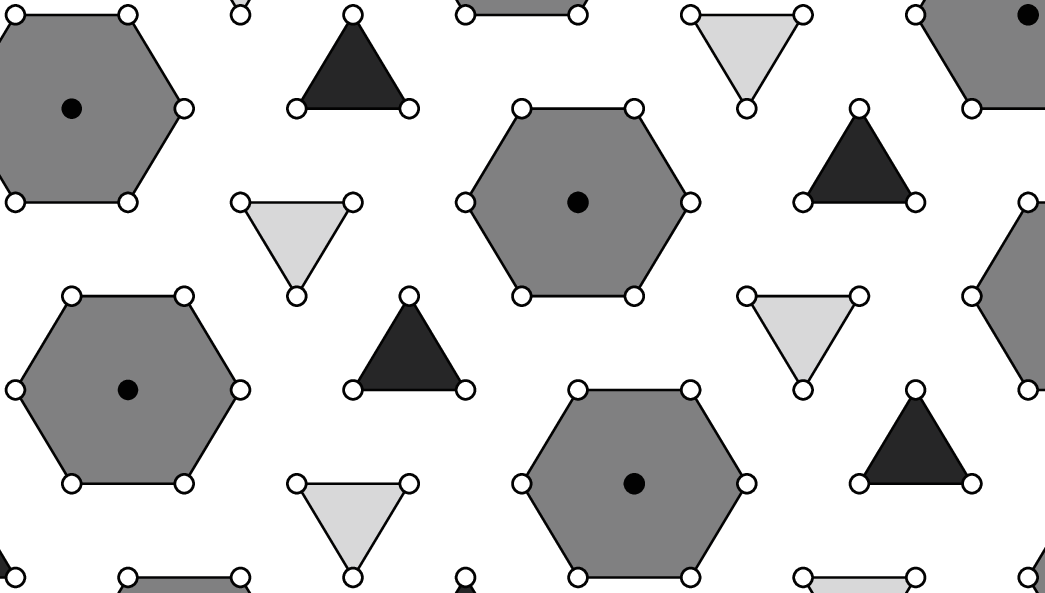}
   \label{fig:A3perfect}
  }
  \subfigure[Intersections of a ball in $ (A_3, d) $ with the planes $ x_0 = \text{const} $.]
  {
	 \makebox[\columnwidth]{
    \includegraphics[width=0.26\columnwidth]{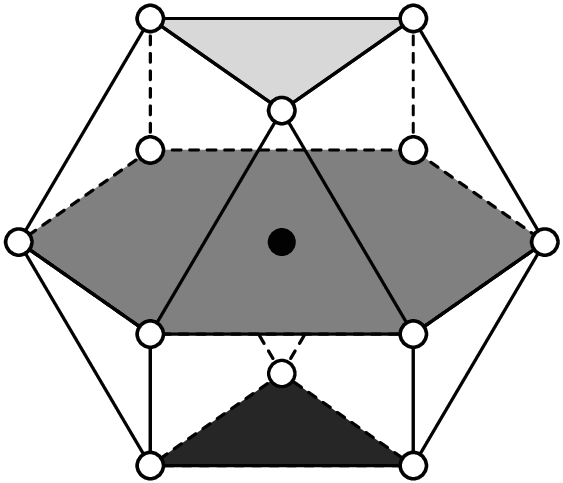}
    \label{fig:ball3color}
	 }
	}
\caption{$ 1 $-perfect code in $ (A_3, d) $.}
\end{figure}%

Another important unsolved problem in the field is the following:
All Abelian planar difference sets live in cyclic groups
\cite[Conj.\ 7.7, p.\ 346]{designs}.
Since the group $ G $ containing the difference set which defines the
code $ \mathcal L $ is isomorphic to $ A_n/{\mathcal L} $, the statement
that $ G $ is cyclic, i.e., that it has a generator, is equivalent to the
following:

\begin{conjecture}[{All Abelian planar difference sets are cyclic}]
\label{conj:cyclic}
  Let $ \mathcal L $ be a linear $ 1 $-perfect code in $ (A_n, d) $.
Then the period of $ \mathcal L $ in $ A_n $ along the direction
$ {\bf f}_{i,j} $ is equal to $ n^2 + n + 1 $ for at least one vector
$ {\bf f}_{i,j} $, $ (i,j) \in \{0, 1, \ldots, n\}^2 $.
\myqed
\end{conjecture}

\subsubsection*{\textbf{\textit{The cyclic case}}}

In the rest of this subsection we restrict our attention to cyclic planar
difference sets of order $ n $, i.e., it is assumed that the group we are
working with is $ \mathbb{Z}_v $, $ v = n^2 + n + 1 $;
as mentioned above, this in fact might not be a restriction at all.
So let $ D = \{ d_0, d_1, \ldots, d_n \} \subset \mathbb{Z}_v $ be a
difference set and $ \mathcal L $ the corresponding code (see \eqref{eq:Bhlattice}).

We shall assume that $ d_0 = 0 $, $ d_1 = 1 $.
(This is not a loss in generality because if $ D $ is a difference set,
there exist two elements, say $ d_0, d_1 \in D $, such that $ d_1 - d_0 = 1 $.
Then we can take the equivalent difference set $ D' = \{ d_i - d_0 : d_i \in D \} $
which obviously contains $ 0 $ and $ 1 $.)
In this case the generator matrix of the code/lattice $ \mathcal L $ has
the following form:
\begin{equation}
 B({\mathcal L}) =
    \begin{pmatrix}
       v      &  0       &  0       &  \cdots  &  0       \\
      -d_2    &  1       &  0       &  \cdots  &  0       \\
      -d_3    &  0       &  1       &  \cdots  &  0       \\
      \vdots  &  \vdots  &  \vdots  &  \ddots  &  \vdots  \\
      -d_n    &  0       &  0       &  \cdots  &  1
    \end{pmatrix} ,
\end{equation}
i.e., the codewords are the vectors $ {\bf x} = \boldsymbol{\xi} \cdot B({\mathcal L}) $,
$ \boldsymbol{\xi} \in \mathbb{Z}^n $ (the vectors are written as rows).
The generator matrix of the dual lattice $ {\mathcal L}^* $ is
\begin{equation}
 B({\mathcal L}^*) = B({\mathcal L})^{-\textsc{t}} =
    \begin{pmatrix}
       \frac{1}{v}    &  \frac{d_2}{v}   &  \frac{d_3}{v}   &  \cdots  &  \frac{d_n}{v}   \\
       0      &  1       &  0       &  \cdots  &  0       \\
       0      &  0       &  1       &  \cdots  &  0       \\
      \vdots  &  \vdots  &  \vdots  &  \ddots  &  \vdots  \\
       0      &  0       &  0       &  \cdots  &  1
    \end{pmatrix} .
\end{equation}
We have disregarded above the $ 0 $-coordinate because $ d_0 = 0 $.
Therefore, $ B({\mathcal L}) $ is in fact a generator matrix of the
corresponding code in $ (\mathbb{Z}^n, \da) $ (see Lemma \ref{th:isometry}).

\subsubsection*{\textbf{\textit{Finite alphabet}}}

By taking the codewords of $ \mathcal L $ modulo $ v = n^2 + n + 1 $, one
obtains a finite code in $ \mathbb{Z}_v^n $ defined by the generator matrix
(over $ \mathbb{Z}_v $)
\begin{equation}
    \begin{pmatrix}
      -d_2    &  1       &  0       &  \cdots  &  0       \\
      -d_3    &  0       &  1       &  \cdots  &  0       \\
      \vdots  &  \vdots  &  \vdots  &  \ddots  &  \vdots  \\
      -d_n    &  0       &  0       &  \cdots  &  1
    \end{pmatrix} .
\end{equation}
This code is of length $ n $, has $ v^{n-1} $ codewords, and is $ 1 $-perfect
(with respect to the obvious ``modulo $ v $ version'' of the $ \da $ metric).
It is also systematic, i.e., the information sequence itself is a part of the
codeword.
The ``parity check'' matrix of the code is
$ H = \begin{pmatrix}  1  &  d_2  &  \cdots  &  d_n  \end{pmatrix} $.
Thus, the codewords are all those vectors $ {\bf x} = (x_1, \ldots, x_n) \in \mathbb{Z}_v^n $
for which $ H \cdot {\bf x}^\textsc{t} = 0 \mod v $, and the syndromes of
the correctable error vectors $ {\bf f}_{i,j} $ (with the $ 0 $-coordinate
left out) are $ H \cdot {\bf f}_{i,j}^\textsc{t} = d_i - d_j $.

\subsection{$ r $-perfect codes in $ (A_n, d) $}
\label{sec:rperfect}

It can be verified directly that, in dimensions $ 1 $ and $ 2 $, $ r $-perfect
codes exist for any $ r $.
Together with Theorem \ref{thm:1perfect} and Singer's construction of
planar difference sets \cite{singer}, we then conclude:

\begin{theorem}
\label{thm:perfect}
  There exists a linear $ r $-perfect code in $ (A_n, d) $ for:
\begin{itemize}
\item
$ n \in \{1, 2\} $, $ r $ arbitrary;
\item
$ n \geq 3 $ a prime power, $ r = 1 $.
\qed
\end{itemize}
\end{theorem}

In higher dimensions, it does not seem possible to tile $ (A_n, d) $ by
balls of radius $ r > 1 $.
The statement that this is indeed the case, which is formulated below as
a conjecture, is a strengthening of the prime power conjecture (Conjecture~\ref{conj:ppc}).
It should also be contrasted with the \emph{Golomb--Welch conjecture}
\cite{golomb+welch, horak+kim} which states that $ r $-perfect codes in
$ \mathbb{Z}^n $ under $ \ell_1 $ metric exist only for:
\begin{inparaenum}[1)]
\item
$ n \in \{1, 2\} $, $ r $ arbitrary, and
\item
$ r = 1 $, $ n $ arbitrary.
\end{inparaenum}
While non-existence has been proven in many cases, the general problem
still remains open; see \cite{horak+kim} and the references therein
(see also \cite{leung+zhou} where the problem has recently been settled
for $ r = 2 $ and all $ n $ in the case of \emph{lattice} packings, i.e.,
linear codes).

\begin{conjecture}
\label{conj:perfect}
$ r $-perfect codes in $ (A_n, d) $ exist only for the pairs $ (n, r) $
listed in Theorem \ref{thm:perfect}.
\myqed
\end{conjecture}

While this is difficult to establish for all the pairs $ (n, r) $, some
cases can be solved by using similar methods to those used in the study
of the Golomb--Welch conjecture, as illustrated by Theorem~\ref{thm:rperfect}
below.

\begin{figure}[h]
 \centering
  \includegraphics[width=0.55\columnwidth]{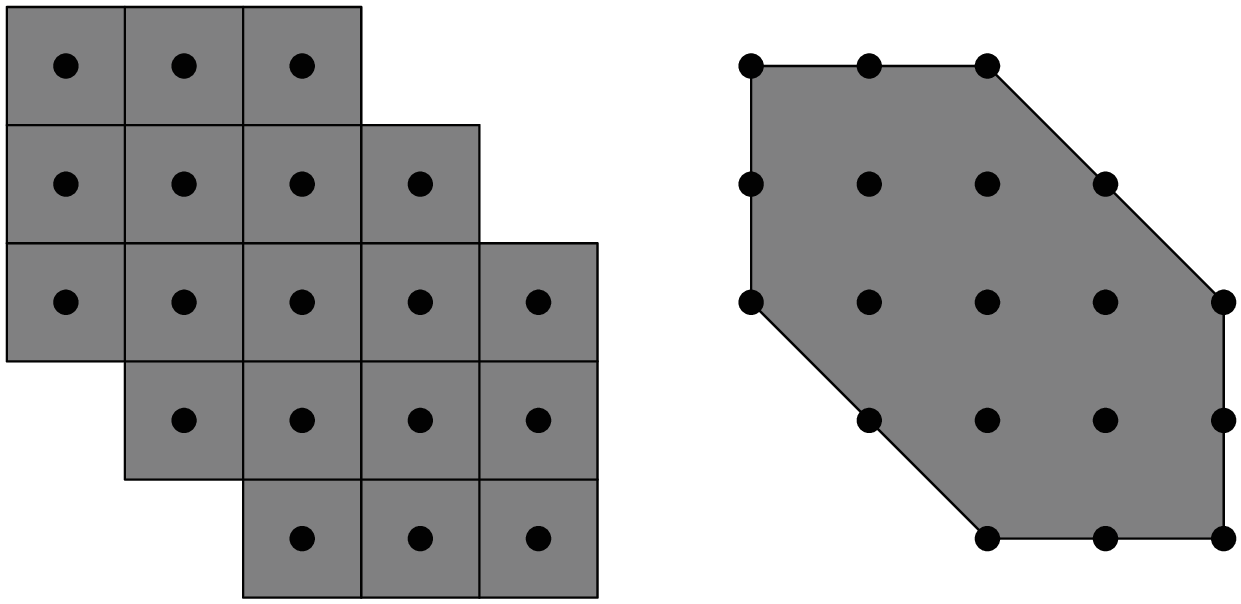}
\caption{Bodies in $ \mathbb{R}^2 $ corresponding to a ball of radius $ 2 $ in $ \left( \mathbb{Z}^2, \da \right) $:
         The cubical tile (left) and the convex interior (right).}
\label{fig:Z2tile}
\end{figure}%

Let $ S_n(r) \equiv S_n(r,r) $ be the ball of radius $ r $ around $ \bf 0 $
in $ ( \mathbb{Z}^n, \da ) $.
Let $ D_n(r) $ be the body in $ \mathbb{R}^n $ defined as the union of unit
cubes translated to the points of $ S_n(r) $, namely,
$ D_n(r) = \bigcup_{{\bf y} \in S_n(r)} ( {\bf y} + [-1/2, 1/2]^n ) $,
and $ C_n(r) $ the body defined as the convex interior in $ \mathbb{R}^n $
of the points in $ S_n(r) $ (see Fig.\ \ref{fig:Z2tile}).

\begin{lemma}
\label{th:volumes}
  The volumes of the bodies $ D_n(r) $ and $ C_n(r) $ are given by
\begin{align}
\label{eq:Dn}
   \vol(D_n(r)) &= \sum_{ m = 0 }^{ \min\{ n, r \} }  { n \choose m }
                                                      { r \choose m }
                                                      { r + n - m \choose n - m } \\ 
   \label{eq:volconv}
   \vol(C_n(r)) &= \frac{r^n}{n!} {2n \choose n} ,
\end{align}
and they satisfy $ \lim_{r \to \infty} \vol(C_n(r)) / \vol(D_n(r)) = 1 $.
\end{lemma}
\begin{proof}
  Since $ D_n(r) $ consists of unit cubes, its volume is $ \vol(D_n(r)) = |S_n(r)| $,
so the expression in \eqref{eq:Dn} follows from Lemma \ref{th:Anball}.

To compute the volume of $ C_n(r) $, consider its intersection with the orthant
$ x_{1}, \ldots, x_{m} > 0 $, $ x_{m+1}, \ldots, x_{n} \leq 0 $, where
$ m \in \{ 0, \ldots, n \} $.
The volume of this intersection is the product of the volumes of the $ m $-simplex
$ \big\{ (x_1, \ldots, x_m) : x_i > 0, \sum x_i \leq r \big\} $, which is known
to be $ r^m/m! $, and of the $ (n-m) $-simplex
$ \big\{ (x_{m+1}, \ldots, x_n) : x_i \leq 0, \sum x_i \geq -r \big\} $, which
is $ r^{n-m}/(n-m)! $.
This implies that
$ \vol(C_n(r)) = \sum_{m=0}^n {n \choose m} \frac{r^m}{m!} \frac{r^{n-m}}{(n-m)!} $,
which is equivalent to \eqref{eq:volconv} since $ \sum_{m=0}^n {n \choose m}^2 = {2n \choose n} $.

To prove the third claim note that if $ r \to \infty $, then
$ {r \choose m} \sim \frac{r^m}{m!} $, and so
$ \vol(D_n(r)) \sim \frac{r^n}{n!} {2n \choose n} = \vol(C_n(r)) $.
\end{proof}

\begin{theorem}
\label{thm:rperfect}
  There are no $ r $-perfect codes in $ (A_n, d) $, $ n \geq 3 $, for large enough
$ r $, i.e., for $ r \geq r_0(n) $.
\end{theorem}
\begin{proof}
 The proof is based on the same idea as the one for $ r $-perfect codes in
$ \mathbb{Z}^n $ under $ \ell_1 $ distance \cite{golomb+welch}.
First note that an $ r $-perfect code in $ ( \mathbb{Z}^n, \da ) $ would
induce a tiling of $ \mathbb{R}^n $ by $ D_n(r) $, and a \emph{packing} by
$ C_n(r) $.
The relative efficiency of the latter with respect to the former is defined as
the ratio of the volumes of these bodies, $ \vol(C_n(r)) / \vol(D_n(r)) $, which
by Lemma \ref{th:volumes} tends to $ 1 $ as $ r $ grows indefinitely.
This has the following consequence: If an $ r $-perfect code exists in
$ (\mathbb{Z}^n, \da) $ for arbitrarily large $ r $, then there exists a
tiling of $ \mathbb{R}^n $ by translates of $ D_n(r) $ for arbitrarily large
$ r $, which further implies that a packing of $ \mathbb{R}^n $ by translates
of $ C_n(r) $ exists which has efficiency arbitrarily close to $ 1 $.
But then there would also be a packing by $ C_n(r) $ of efficiency $ 1 $, i.e., a
tiling (in \cite[Appendix]{golomb+welch} it is shown that there exists a packing
whose density is the supremum of the densities of all possible packings with a
given body).
This is a contradiction.
Namely, Minkowski \cite{minkowski} (see also \cite[Thm 1]{mcmullen}) has shown
that a necessary condition for a convex body to be able to tile space is that
it be a polytope with centrally symmetric%
\footnote{A polytope $ P \subset \mathbb{R}^n $ is centrally symmetric if
its translation $ \tilde{P} = P - {\bf x} $ satisfies $ \tilde{P} = -\tilde{P} $
for some $ {\bf x} \in \mathbb{R}^n $.}
facets, which $ C_n(r) $ fails to satisfy for $ n \geq 3 $.
For example, the facet which is the intersection of $ C_n(r) $ with the hyperplane
$ x_1 = -r $ is the simplex $ \big\{ (x_2, \ldots, x_n) : x_i \geq 0, \sum_{i=2}^n x_i \leq r \big\} $,
a non-centrally-symmetric body.
\end{proof}

\subsection{Tilings by $ S_n(r,r-1) $ and diameter-perfect codes}

In analogy with $ r $-perfect codes and perfect $ B_{2r} $ sets (Section \ref{sec:rperfect}),
one may define perfect $ B_{2r-1} $ sets as those that give rise to tilings
of $ \Z^n $ by the sets $ S_n(r, r-1) $.
In dimension $ 1 $ the problem is trivial, and in dimension $ 2 $ the tilings
exist for any $ r \geq 1 $ (see Fig.~\ref{fig:S2_tiling}).
The tiling lattice $ {\mathcal L} $ for the set $ S_2(r, r-1) $ is the one
spanned by the vectors $ (r, r) $ and $ (0, 3r) $, and it is unique, which
can be seen from the figure.
It should be noted that, for $ r \geq 2 $, $ \mathbb{Z}^2 / {\mathcal L} $
is not cyclic and hence, perfect $ B_{2r-1} $ sets of cardinality $ 3 $ do
not exist in cyclic groups.
For $ r = 1 $, tiling of $ \mathbb{Z}^n $ by $ S_n(1,0) $ exists for any $ n $;
it corresponds to the  trivial $ B_1 $ set $ G $ in an arbitrary Abelian group
$ G $.

The following generalization of a notion of perfect code was introduced in \cite{ahlswede}.
We say that a code $ \C \subseteq \mathbb{Z}^n $ of minimum distance $ \da(\C) $
is \emph{diameter-perfect} if there exists a set (anticode) $ S \subset \mathbb{Z}^n $
of diameter $ \da(\C) - 1 $ whose translates to the points in $ \C $ form a tiling
of $ \Z^n $.
This notion is especially interesting when the minimum distance of a code is even,
which can never be the case for perfect codes.
By the observations from the previous paragraph, we have:

\begin{theorem}
\label{thm:diamperfect}
  There exists a linear diameter-perfect code of minimum distance $ 2 r $ in
$ (\mathbb{Z}^n, \da) $ for:
\begin{itemize}
\item
$ n \in \{1, 2\} $, $ r $ arbitrary;
\item
$ n \geq 3 $, $ r = 1 $.
\qed
\end{itemize}
\end{theorem}

\begin{conjecture}
\label{conj:diamperfect}
Diameter-perfect codes of minimum distance $ 2 r $ in $ (\mathbb{Z}^n, \da) $
exist only for the pairs $ (n, r) $ listed in Theorem \ref{thm:diamperfect}.
\myqed
\end{conjecture}

In dimensions $ n \geq 3 $, a statement analogous to Theorem~\ref{thm:rperfect}
can be proven to exclude the existence of tilings by $ S_n(r, r-1) $ for $ r \geq r_1(n) $.

\section{$ (v, k, \lambda) $-difference sets and coverings of $ A_n $}
\label{sec:gendif}

There are several generalizations of difference sets and $ B_h $ sets which
can be interpreted geometrically using the same methods as above.
We mention here one such extensively studied notion \cite{designs}, that of
a $ (v, k, \lambda) $-difference set, which generalizes planar difference sets
studied in the previous section.
Let $ G $ be a group of order $ v $, as before.
A set $ D \subseteq G $ of cardinality $ k $ is said to be a
$ (v, k, \lambda) $-\emph{difference set} if every nonzero element of $ G $
can be expressed as a difference $ d_i - d_j $ of two elements from $ D $
in exactly $ \lambda $ ways.
The parameters $ v, k, \lambda $ then necessarily satisfy the identity
$ \lambda(v-1) = k(k-1) $.
The \emph{order} of such a difference set is defined as $ k - \lambda $.
Planar difference sets are obtained for $ \lambda = 1 $.

\subsection*{Geometry of Abelian $ (v, k, \lambda) $-difference sets}

In the following, when using concepts from graph theory in our setting, we
have in mind the graph representation $ \Gamma(A_n) $ of $ A_n $, as introduced
in Section \ref{sec:An}.
An $ (r, i, j) $-cover (or $ (r, i, j) $-covering code) in a graph $ \Gamma = (V, E) $
\cite{axenovich} is a set of its vertices $ S \subseteq V $ with the property
that every element of $ S $ is covered by exactly $ i $ balls of radius $ r $
centered at elements of $ S $, while every element of $ V \setminus S $ is
covered by exactly $ j $ such balls.
Special cases of such sets, namely $ (1, i, j) $--covers, have also been
studied in the context of domination theory in graphs \cite{telle};
in coding theory, $ (r, 1, 1) $-covers are known as $ r $-perfect codes.
An independent set in a graph $ \Gamma = (V, E) $ is a subset of its vertices
$ I \subseteq V $, no two of which are adjacent in $ \Gamma $.

The proof of the following theorem is an easy generalization of the connection
between lattice tilings and group splitting used in the previous sections
\cite{stein67, hamaker, stein74, galovich, hickerson} (see also \cite{stein+szabo}).
We write it nonetheless for completeness.%

\begin{theorem}
\label{th:domination}
  There exists an Abelian $ (v, n+1, \lambda) $-difference set if and only if
the lattice $ A_{n} $ contains a $ (1, 1, \lambda) $-covering sublattice.
\end{theorem}
\begin{proof}
  Suppose that $ D = \{ d_0, d_1, \ldots, d_{n} \} $ is a $ (v, n+1, \lambda) $-%
difference set in an Abelian group $ G $, and consider the sublattice
\begin{equation}
\label{eq:code}
  {\mathcal L} = \left\{ {\bf x} \in A_{n} : \sum_{i=0}^{n} x_i d_i = 0 \right\} .
\end{equation}
$ {\mathcal L} $ is a $ (1, 1, \lambda) $-cover in $ A_{n} $.
To see this, consider a point $ {\bf y} = (y_0, y_1, \ldots, y_{n}) \notin {\mathcal L} $,
meaning that $ \sum_{i=0}^n y_i d_i = a \in G $, $ a \neq 0 $.
The neighbors of $ \bf y $ are of the form $ {\bf y} + {\bf f}_{i,j} $, $ i \neq j $.
Since $ D $ is a difference set, $ -a \in G $ can be written as a difference of
the elements from $ D $ in exactly $ \lambda $ ways, meaning that there are
$ \lambda $ different pairs $ (s,t) $ for which $ d_s - d_t = -a $, $ d_s, d_t \in D $.
For every such pair consider the point $ {\bf z}_{s,t} = {\bf y} + {\bf f}_{s,t} $.
$ {\bf z}_{s,t} \in {\mathcal L} $ because
$ \sum_{i=0}^n z_i d_i = \sum_{i=0}^n y_i d_i + d_s - d_t = a - a = 0 $.
Therefore, there are exactly $ \lambda $ points in the lattice $ \mathcal L $ that
are adjacent to $ \bf y $, i.e., such that balls of radius $ 1 $ around them cover
$ \bf y $.
To show that the elements of $ {\mathcal L} $ are covered only by the balls around
themselves (i.e., that $ \mathcal L $ is an independent set in $ \Gamma(A_n) $),
note that if there were two points at distance $ 1 $ in $ {\mathcal L} $, then by
the same argument as above we would obtain that $ d_s - d_t = 0 $, i.e., $ d_s = d_t $
for some $ s \neq t $, which is not possible if $ | D | = n+1 $.

For the other direction, assume that $ {\mathcal L}' $ is a $ (1, 1, \lambda) $-covering
sublattice of $ A_{n} $.
Consider the quotient group $ G = A_{n} / {\mathcal L}' $, and take
$ D = \{ d_0, d_1, \ldots, d_{n} \} \subseteq G $, where
$ d_i = [{\bf f}_{i,0}] \equiv {\bf f}_{i,0} + {\mathcal L}' $ are cosets
(elements of $ G $).
Let us first assure that all the $ d_i $'s are distinct.
Suppose that $ d_s = d_t $ for some $ s \neq t $.
This implies that $ d_s - d_t = [{\bf f}_{s,t}] = [{\bf 0}] $, which means that
$ {\bf f}_{s,t} \in {\mathcal L}' $.
But since $ {\bf 0} \in {\mathcal L}' $, and $ \bf 0 $ and $ {\bf f}_{s,t} $ are
at distance $ 1 $, this would contradict the fact that $ {\mathcal L}' $ is independent.
Hence, $ | D | = n+1 $.
Now take any nonzero element of $ G $, say $ [{\bf y}] $, $ {\bf y} \notin {\mathcal L}' $.
By assumption, $ \bf y $ is covered by exactly $ \lambda $ elements of $ {\mathcal L}' $,
i.e., $ {\bf y} + {\bf f}_{s,t} \in {\mathcal L}' $ for exactly $ \lambda $ vectors
$ {\bf  f}_{s,t} $.
Since $ {\bf f}_{s,t} = {\bf f}_{s,0} - {\bf f}_{t,0} $, this means that
$ d_t - d_s = [{\bf f}_{t,0}] - [{\bf f}_{s,0}] = [{\bf y}] $ for exactly $ \lambda $
pairs $ (s,t) $.
$ D $ is therefore a $ (v,n+1,\lambda) $-difference set.
\end{proof}

Geometrically, the theorem states that balls of radius $ 1 $ around the points of
the sublattice $ {\mathcal L} $ overlap in such a way that every point that does
not belong to $ {\mathcal L} $ is covered by exactly $ \lambda $ balls.
(The points in $ {\mathcal L} $ -- centers of the balls -- are covered by one
ball only, and hence this notion is different from multitiling \cite{multitiling}.)

\begin{example}
  $ D = \{ 0, 1, 2 \} $ is a $ (4, 3, 2) $-difference set in the cyclic group
$ \mathbb{Z}_4 $.
A $ (1, 1, 2) $-covering sublattice $ {\mathcal L} \subset A_2 $ corresponding
to this difference set is illustrated in Fig.~\ref{fig:domination}.
Points in $ {\mathcal L} $ are depicted as black, and those in $ A_2 \setminus {\mathcal L} $
as white dots.
For illustration, Fig.~\ref{fig:dominationnotI} shows an example of a $ (1, 3, 2) $-covering
sublattice, which does not correspond to any difference set.
\myqed
\end{example}

\begin{figure}[h]
\centering
  \includegraphics[width=0.77\columnwidth]{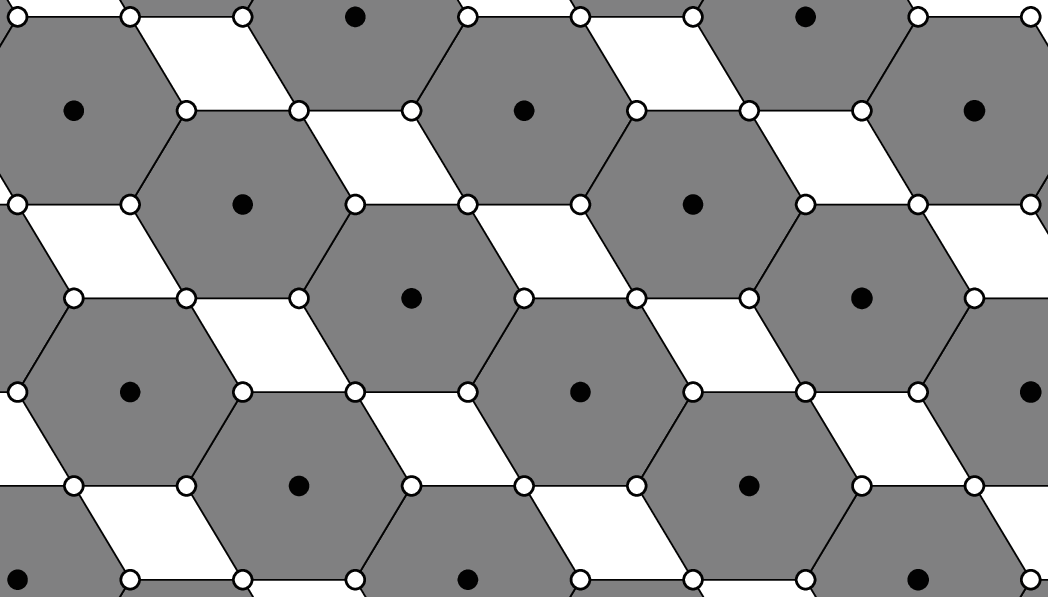}
  \caption{A $ (1, 1, 2) $-covering sublattice of $ A_2 $.}
\label{fig:domination}
\end{figure}%

\begin{figure}[h]
\centering
  \includegraphics[width=0.77\columnwidth]{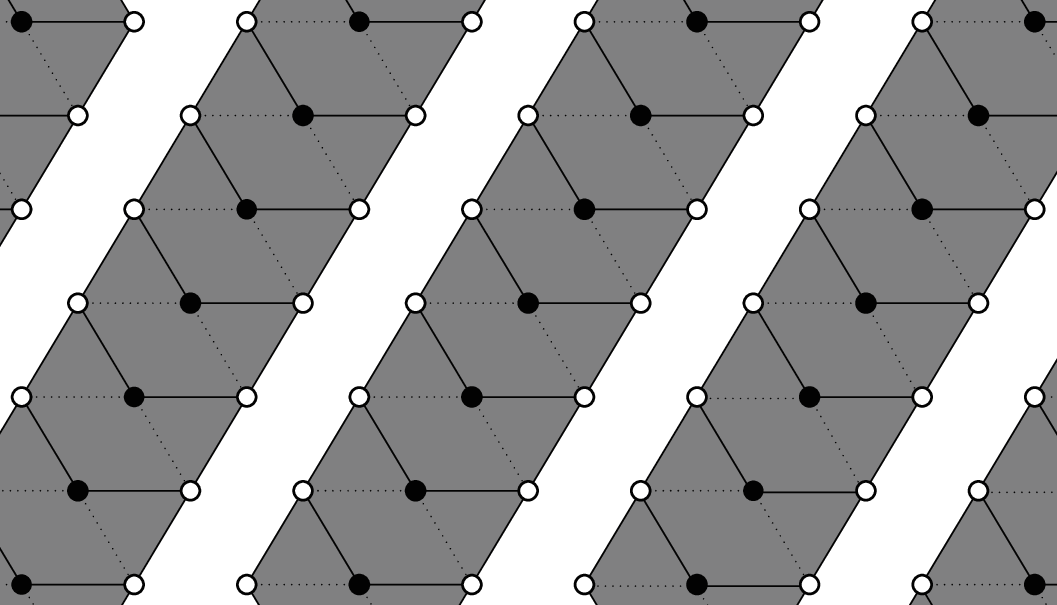}
  \caption{A $ (1, 3, 2) $-covering sublattice of $ A_2 $.}
\label{fig:dominationnotI}
\end{figure}%

\begin{remark}
  For $ \lambda = 1 $, the code $ \mathcal L $ from \eqref{eq:code} can correct
a single error because balls of radius $ 1 $ around codewords do not overlap and
the minimum distance of the code is $ 3 $ (here by a single error we mean the addition
of a vector $ {\bf f}_{i,j} $ for some $ i, j $, $ i \neq j $, to the ``transmitted''
codeword $ {\bf x} \in {\mathcal L} $).
For $ \lambda > 1 $, however, it can only \emph{detect} a single error reliably.
Note also that increasing $ \lambda $ increases the density of the code/lattice
$ \mathcal L $ in $ A_{n} $, but does not affect its error-detection capability.
The densest such lattice is therefore obtained for $ \lambda = n - 1 $ (that this
is the maximum value follows from $ \lambda(v-1) = n(n+1) $ and $ n \leq v - 1 $);
it corresponds to the trivial $ (v,v,v) $-difference set $ D = G $ in an arbitrary
Abelian group $ G $.
\myqed
\end{remark}

Note that, as in \eqref{eq:Bhlattice}, we have not specified the order of the
elements of  $ D $ when defining the corresponding lattice $ \mathcal L $ in
\eqref{eq:code} because it would affect it in an insignificant way only.
Note also that if we write $ d_i' = zd_i + g $ instead of $ d_i $ in \eqref{eq:code},
where $ z $ is a fixed integer coprime with $ v $ and $ g $ is a fixed element
of $ G $, the same lattice is obtained because
\begin{equation}
 \sum_{i=0}^{n} x_i d_i = 0 \quad \Leftrightarrow \quad \sum_{i=0}^{n} x_i d_i' = 0
\end{equation}
which follows from $ \sum_{i=0}^{n} x_i = 0 $ and $ \operatorname{gcd}(z, v) = 1 $.

Let us recall some terminology.
Two difference sets $ D $ and $ D' $ in an Abelian group $ G $ are said to be
equivalent \cite[Rem.\ 1.11, p.\ 302]{designs} if $ D' = \{ zd + g: d \in D \} $,
for some $ z \in \mathbb{Z} $ coprime with $ v $ and some $ g \in G $.
Two codes $ \mathcal C $ and $ {\mathcal C}' $ of length $ m $ over an alphabet
$ \mathbb{A} $ are equivalent \cite[p.\ 40]{macwilliams+sloane} if there exist
$ m $ permutations of $ \mathbb{A} $, $ \pi_1, \ldots, \pi_m $, and a permutation
$ \sigma $ over $ \{ 1, \ldots, m \} $ such that
\begin{equation}
 {\mathcal C}' = \big\{ \sigma( \pi_1(x_1), \ldots, \pi_m(x_m) ) : (x_1, \ldots, x_m) \in {\mathcal C} \big\} .
\end{equation}
We then have the following:

\begin{proposition}
 If two difference sets $ D $ and $ D' $ are equivalent, then the corresponding
codes (defined as in \eqref{eq:code}) are equivalent.
\hfill \qed
\end{proposition}
\noindent
In fact, the $ \pi_i $'s are necessarily identity maps, only $ \sigma $ is
relevant here.

\section{Concluding remarks}

Codes in $ A_n $ lattices and closely related packings in $ (\mathbb{Z}^n, \da) $
are objects whose study leads to interesting and challenging problems at the
intersection of algebra, combinatorics, geometry, and coding theory, and which
are also well-motivated and applicable in various information storage and
transmission scenarios.
We have presented in this chapter certain aspects of these problems, primarily
in the linear case, such as their connection to Sidon sets and difference sets,
bounds on the dimension $ n $ and the packing radius $ r $, as well as the
existence of perfect codes.
Several conjectures have been formulated along the way as a guide for further
work on the subject.
Apart from these conjectures, we mention here another interesting problem that
has not received much attention but that is very important both from the
theoretical perspective and for the intended applications.
Namely, in cases when perfect packings do not exist, it is of interest to
determine the maximum packing density of codes having a given radius (or
minimum distance).
Determining this quantity exactly is at present too ambitious a goal for most
parameters $ n, r $;
however, deriving good bounds on the maximum packing density and studying its
behavior in various asymptotic regimes is attainable and represents, in our
opinion, a worthwhile research direction in this context.

\section*{Acknowledgment}

I would like to use this opportunity to thank the organizers and the participants
of the International Seminar on Algebra and Coding Theory (INSACT 2017), where
some of this work was presented, in particular the faculty members and the students
in the Department of Mathematics, St Berchmans College, Changanacherry, Kerala,
India.
I am very grateful to Shine C Mathew and Antony Mathews, for the organization
and implementation of this wonderful event, and to my former colleague Eldho K
Thomas, for inviting me to Kerala.
I am indebted to all of them for their extraordinary hospitality.
Finally, I would like to thank my former postdoc advisor Vincent Y. F. Tan
(National University of Singapore), with whom this work was further developed,
which resulted in our joint publications \cite{kovacevic+tan_sidma, kovacevic+tan_it}.

\bibliographystyle{amsplain}

\begin{thebibliography}{99}

\bibitem{ahlswede}
   R. Ahlswede, H. K. Aydinian, and L. H. Khachatrian,
   ``On Perfect Codes and Related Concepts,''
   \emph{Des. Codes Cryptogr.}, vol.~22, no.~3, pp.~221--237, 2001.
\bibitem{atkinson}
   M. D. Atkinson, N. Santoro, and J. Urrutia,
   ``Integer Sets with Distinct Sums and Differences and Carrier Frequency Assignments for Nonlinear Repeaters,''
   \emph{IEEE Trans. Commun.}, vol.~34, no.~6, pp.~614--617, 1986.
\bibitem{axenovich}
   M. A. Axenovich,
   ``On Multiple Coverings of the Infinite Rectangular Grid with Balls of Constant Radius,''
   \emph{Discrete Math.}, vol.~268, no.~1--3, pp.~31--48, 2003.
\bibitem{barg}
   A. Barg and A. Mazumdar,
   ``Codes in Permutations and Error Correction for Rank Modulation,''
   \emph{IEEE Trans. Inform. Theory}, vol.~56, no.~7, pp.~3158--3165, 2010.
\bibitem{designs}
   T. Beth, D. Jungnickel, and H. Lenz,
   \emph{Design Theory}, 2nd ed.,
   Cambridge University Press, 1999.
\bibitem{bose+chowla}
   R. C. Bose and S. Chowla,
   ``Theorems in the Additive Theory of Numbers,''
   \emph{Comment. Math. Helv.}, vol.~37, no.~1, pp.~141--147, 1962.
\bibitem{chen}
   S. Chen,
   ``On the Size of Finite {S}idon Sequences,''
   \emph{Proc. Amer. Math. Soc.}, vol.~121, no.~2, pp.~353--356, 1994.
\bibitem{conway+sloane}
   J. H. Conway and N. J. A. Sloane,
   \emph{Sphere Packings, Lattices and Groups},
   3rd ed., Springer, 1999.
\bibitem{delsarte}
   P. Delsarte,
   ``An Algebraic Approach to Association Schemes of Coding Theory,''
   \emph{Philips J. Res.}, vol.~10, pp.~1--97, 1973.
\bibitem{derksen}
   H. Derksen,
   ``Error-Correcting Codes and $ B_h $-Sequences,''
   \emph{IEEE Trans. Inform. Theory}, vol.~50, no.~3, pp.~476--485, 2004.
\bibitem{ding}
   C. Ding,
   \emph{Codes from Difference Sets},
   World Scientific, 2015.
\bibitem{fejes-toth}
   G. Fejes T\'{o}th,
   ``Packing and Covering,''
   in: J. E. Goodman and J. O'Rourke (eds.), \emph{Handbook of Discrete and Computational Geometry}, 2nd ed., CRC Press, 2004.
\bibitem{galovich}
   S. Galovich and S. Stein,
   ``Splittings of {A}belian Groups by Integers,''
   \emph{Aequationes Math.}, vol.~22, no.~1, pp.~249--267, 1981.
\bibitem{golomb+welch}
   S. W. Golomb and L. R. Welch,
   ``Perfect Codes in the {L}ee Metric and the Packing of Polyominoes,''
   \emph{SIAM J. Appl. Math.}, vol.~18, no.~2, pp.~302--317, 1970.
\bibitem{graham+sloane}
   R. L. Graham and N. J. A. Sloane,
   ``Lower Bounds for Constant Weight Codes,''
   \emph{IEEE Trans. Inform. Theory}, vol.~26, no.~1, pp.~37--43, 1980.
\bibitem{multitiling}
   N. Gravin, S. Robins, and D. Shiryaev,
   ``Translational Tilings by a Polytope, with Multiplicity,''
   \emph{Combinatorica}, vol.~32, no.~6, pp.~629--648, 2012.
\bibitem{hamaker}
   W. Hamaker,
   ``Factoring Groups and Tiling Space,''
   \emph{Aequationes Math.}, vol.~9, no.~2-3, pp.~145--149, 1973.
\bibitem{hickerson}
   D. Hickerson,
   ``Splittings of Finite Groups,''
   \emph{Pacific J. Math.}, vol.~107, no.~1, pp.~141--171, 1983.
\bibitem{horak+kim}
   P. Horak and D. Kim,
   ``50 Years of the {G}olomb--{W}elch Conjecture,''
   \emph{IEEE Trans. Inform. Theory}, vol.~64, no.~4, pp.~3048--3061, 2018.
\bibitem{jia}
   X.-D. Jia,
   ``On Finite {S}idon Sequences,''
   \emph{J. Number Theory}, vol.~44, no.~1, pp.~84--92, 1993.
\bibitem{klove}
   T. Kl\o ve,
   ``Error Correcting Codes for the Asymmetric Channel,''
   Technical Report, Dept. of Informatics, University of Bergen, 1981. (Updated in 1995.)
\bibitem{kovacevic2}
   M. Kova\v{c}evi\'c,
   ``Runlength-Limited Sequences and Shift-Correcting Codes: Asymptotic Analysis,''
   \emph{IEEE Trans. Inform. Theory}, vol.~65, no.~8, pp.~4804--4814, 2019.
\bibitem{kovacevic+tan_sidma}
   M. Kova\v{c}evi\'c and V. Y. F. Tan,
   ``Improved Bounds on {S}idon Sets via Lattice Packings of Simplices,''
   \emph{SIAM J. Discrete Math.}, vol.~31, no.~3, pp.~2269--2278, 2017.
\bibitem{kovacevic+tan_it}
   M. Kova\v{c}evi\'c and V. Y. F. Tan,
   ``Codes in the Space of Multisets---Coding for Permutation Channels with Impairments,''
   \emph{IEEE Trans. Inform. Theory}, vol.~64, no.~7, pp.~5156--5169, 2018.
\bibitem{kovacevic+tan_clet}
   M. Kova\v{c}evi\'c and V. Y. F. Tan,
   ``Asymptotically Optimal Codes Correcting Fixed-Length Duplication Errors in {DNA} Storage Systems,''
   \emph{IEEE Commun. Lett.}, vol.~22, no.~11, pp.~2194--2197, 2018.
\bibitem{kovacevic+vukobratovic_desi}
   M. Kova\v{c}evi\' c  and D. Vukobratovi\'c,
   ``Perfect Codes in the Discrete Simplex,''
   \emph{Des. Codes Cryptogr.}, vol.~75, no.~1, pp.~81--95, 2015.
\bibitem{lam+sarwate}
   A. W. Lam and D. V. Sarwate,
   ``On Optimum Time-Hopping Patterns,''
   \emph{IEEE Trans. Commun.}, vol.~36, no.~3, pp.~380--382, 1988.
\bibitem{leung+zhou}
   K. H. Leung and Y. Zhou,
   ``No Lattice Tiling of $ \mathbb{Z}^n $ by {L}ee Sphere of Radius 2,''
   \emph{J. Combin. Theory Ser. A}, vol.~171, article no.~105157, 2020.
\bibitem{macwilliams+sloane}
   F. J. MacWilliams and N. J. A. Sloane,
   \emph{The Theory of Error-Correcting Codes},
   North-Holland Publishing Company, 1977.
\bibitem{mcmullen}
   P. McMullen,
   ``Convex Bodies Which Tile Space by Translation,''
   \emph{Mathematika}, vol.~27, no.~1, pp.~113--121, 1980.
\bibitem{minkowski}
   H. Minkowski,
   ``Allgemeine Lehrs\"{a}tze \"{u}ber Convexen Polyeder'' (in German),
   \emph{Nachr. K. Acad. Wiss. G\"{o}ttingen, Math.-Phys. Kl. ii}, pp.~198--219, 1897.
\bibitem{obryant}
   K. O'Bryant,
   ``A Complete Annotated Bibliography of Work Related to {S}idon Sequences,''
   \emph{Electron. J. Combin.}, \#DS11, 39~p. (electronic), 2004.
\bibitem{sidon}
   S. Sidon,
   ``Ein Satz \"{u}ber Trigonometrische Polynome und Seine Anwendung in der Theorie der {F}ourier-Reihen'' (in German),
   \emph{Math. Ann.}, vol.~106, no.~1, pp.~536--539, 1932.
\bibitem{singer}
   J. Singer,
   ``A Theorem in Finite Projective Geometry and Some Applications to Number Theory,''
   \emph{Trans. Amer. Math. Soc.}, vol.~43, pp.~377--385, 1938.
\bibitem{stein67}
   S. Stein,
   ``Factoring by Subsets,''
   \emph{Pacific J. Math.}, vol.~22, no.~3, pp.~523--541, 1967.
\bibitem{stein74}
   S. Stein,
   ``Algebraic Tiling,''
   \emph{Amer. Math. Monthly}, vol.~81, pp.~445--462, 1974.
\bibitem{stein84}
   S. Stein,
   ``Packings of $ R^n $ by Certain Error Spheres,''
   \emph{IEEE Trans. Inform. Theory}, vol.~30, no.~2, pp.~356--363, 1984.
\bibitem{stein+szabo}
   S. Stein and S. Szab\'{o},
   \emph{Algebra and Tiling: Homomorphisms in the Service of Geometry},
   The Mathematical Association of America, 1994.
\bibitem{telle}
   J. A. Telle,
   ``Complexity of Domination-Type Problems in Graphs,''
   \emph{Nordic J. Comput.}, vol.~1, pp.~157--171, 1994.
\bibitem{varshamov}
   R. R. Varshamov,
   ``A Class of Codes for Asymmetric Channels and a Problem from the Additive Theory of Numbers,''
   \emph{IEEE Trans. Inform. Theory}, vol.~19, no.~1, pp.~92--95, 1973.

\end{thebibliography}

\end{document}